\documentclass[11pt]{amsart}
\usepackage[mathlines, displaymath]{lineno}
\usepackage[margin=1.3in]{geometry}
\usepackage{amscd,amssymb, amsmath,  setspace, wasysym, yfonts, amsthm, mathtools, dirtytalk}
\usepackage{graphicx}
\usepackage[all]{xy}
\usepackage{tikz}
\usetikzlibrary{matrix}

\usepackage{url}
\usepackage{hyperref}

\theoremstyle{plain}
\newtheorem{theorem}{Theorem}

\newtheorem{cor}[theorem]{Corollary}
\newtheorem{conj}[theorem]{Conjecture}

\newtheorem{lemma}[theorem]{Lemma}

\theoremstyle{definition}

\newtheorem{defn}[theorem]{Definition}

\newtheorem{question}[theorem]{Question}

\newtheorem*{ex*}{Example}

\DeclareMathOperator{\Pic}{Pic}

\DeclareMathOperator{\Aut}{Aut}

\DeclarePairedDelimiterX\set[1]\lbrace\rbrace{#1}

\def\D{\mathbf{D}}
\def\Aut{{\rm Aut}}
\def\Pic{{\rm Pic}}
\def\OO{{\mathcal O}}
\def\ZZ{{\mathbf Z}}
\def\E{\mathcal{E}}
\def\R{\mathbf{R}}
\def\id{{\rm id}}
\def\F{\mathcal{F}}
\def\NN{{\mathbf N}}
\def\I{\mathcal{I}}

\newenvironment{dedication}
  {\thispagestyle{empty}
   \vspace{0.5em}
   \itshape             
        \centering     
  }
  {\par 
   \vspace{1em} 
        
  }

\title[Derived equivalence and non-vanishing loci II]
{Derived equivalence and non-vanishing loci II}

\author[L.~Lombardi]{Luigi Lombardi}
\address{Department of Mathematics, Stony Brook University, Stony Brook, NY 11794-3651}
\email{luigi.lombardi@stonybrook.edu}

\author[M.~Popa]{Mihnea Popa }
\address{Department of Mathematics, Northwestern University,
2033 Sheridan Road, Evanston, IL 60208, USA} 
\email{mpopa@math.northwestern.edu}
\begin{document}

 \maketitle

\begin{dedication}
Dedicated to Rob Lazarsfeld on the occasion of his sixtieth birthday, with warmth and gratitude.
\end{dedication}

\begin{abstract}
We prove a few cases of a conjecture on the invariance of cohomological support loci under derived equivalence
by establishing a concrete connection with the related problem of the invariance of Hodge numbers. We use the 
main case in order to study the derived behavior of fibrations over curves.
\end{abstract}

\section{Introduction}

This paper is concerned with the following conjecture made in \cite{popa} on the behavior of the
non-vanishing loci for the cohomology of deformations of the canonical bundle under derived equivalence.  
We recall that given a smooth projective $X$ these loci, more commonly called cohomological support loci, 
are the closed algebraic subsets of the Picard variety defined as
$$V^i( \omega_X) : = \{ \alpha ~|~ H^i (X, \omega_X \otimes \alpha) \neq 0\} 
\subseteq {\rm Pic}^0 (X).$$
All varieties we consider are defined over the complex numbers.
We denote by $\D(X)$ the bounded derived category of coherent sheaves 
$\D^{{\rm b}} ({\rm Coh}(X))$.  

\begin{conj}[\cite{popa}]\label{var1}
Let $X$ and $Y$ be smooth projective varieties with $\D(X) \simeq \D(Y)$ as triangulated categories. Then 
$$V^i (\omega_X)_0  \simeq V^i (\omega_Y)_0 \,\,\,\, {\rm for ~all~} i \ge 0,$$ 
where $V^i (\omega_X)_0$ denotes the union of the irreducible components of 
$V^i (\omega_X)$ passing through the origin, and similarly for $Y$.
\end{conj}

We refer to \cite{popa} and \cite{lombardi} for a general discussion of 
this conjecture and its applications, and of the cases in which it has been known to hold (recovered below as well).
The main point of this paper is to directly relate Conjecture \ref{var1} 
to part of the well-known problem of the invariance of Hodge numbers under derived equivalence; we state only the special case we need.

\begin{conj}\label{special_hodge}
Let $X$ and $Y$ be smooth projective varieties with 
$\D(X) \simeq \D(Y)$. Then 
$$h^{0,i} (X) = h^{0,i} (Y) \,\,\,\, {\rm for ~all~} i \ge 0.$$ 
\end{conj}

Our main result is the following:

\begin{theorem}\label{main}
Conjecture \ref{special_hodge} implies Conjecture \ref{var1}. 
More precisely, Conjecture \ref{var1} for a given $i$ is implied by Conjecture \ref{special_hodge} for $n-i$, where
$n = {\rm dim}~X$.
\end{theorem}

This leads to a verification of Conjecture \ref{var1} in a few important cases, corresponding 
to the values of $i$ for which Conjecture \ref{special_hodge} is already known to hold.

\begin{cor}\label{non-vanishing} 
Let $X$ and $Y$ be smooth projective varieties of dimension $n$, with $\D(X) \simeq \D(Y)$. Then
$$V^i (\omega_X)_0 \simeq V^i (\omega_Y)_0\,\,\,\,{\rm for}\,\,\,\, i = 0, 1, n-1, n.$$
\end{cor}
\begin{proof}
According to Theorem \ref{main}, we need to know that derived equivalence implies the 
invariance of $h^{0,n}$, $h^{0, n-1}$ and $h^{0,1}$. The first two are well-known consequences of the invariance of Hochschild homology, while
the last is the main result of \cite{PS}.
\end{proof}

A stronger result than Theorem \ref{main} and Corollary \ref{non-vanishing}, involving the dimension of cohomology groups 
related via the isomorphism, is in fact proved in \S\ref{sec comparison} (see Conjecture \ref{stronger} 
and Theorem \ref{main_improved}). For $i = 0,1$ this was proved in \cite{lombardi} by means of a 
twisted version of Hochschild homology. We also note that the Corollary 
above recovers a result first proved in \cite[\S4]{lombardi}, namely that Conjecture \ref{var1} holds for varieties of dimension up to $3$. 
We can also conclude that it holds for an important class of irregular fourfolds.

\begin{cor}
Conjecture \ref{var1} holds in dimension up to three, and for fourfolds of maximal Albanese dimension.
\end{cor} 
\begin{proof}
The first part follows immediately from Corollary \ref{non-vanishing}.
For the second, 
according to Corollary \ref{non-vanishing} and Theorem \ref{main} it suffices to have $h^{0,2} (X) = h^{0,2} (Y)$, 
which is proved for derived equivalent fourfolds of maximal Albanese dimension in \cite[Corollary 1.8]{lombardi}.\footnote{Note that the 
same holds for fourfolds of non-negative Kodaira dimension whose Albanese image has dimension $3$, and for those with non-affine $\Aut^0 (X)$.}
\end{proof}

Overall, besides the unified approach, from the point of view of Conjecture \ref{var1} the key new result and applications here are in the 
case $i = n-1$. By theorems of Beauville \cite{beauville} and Green-Lazarsfeld \cite{gl}, the cohomological support loci 
$V^{n-1} (\omega_X)$ are the most ``geometric" among the $V^i$, corresponding in a quite precise way to fibrations of $X$ over curves. 
This leads to the following structural application; note that while Fourier-Mukai equivalences between smooth projective surfaces are 
completely classified (\cite{BM2}, \cite{kawamata}), in higher dimension few results towards classification are available (see e.g. \cite{toda}).

\begin{theorem}\label{pencils}
Let $X$ and $Y$ be smooth projective varieties with $\D(X) \simeq \D(Y)$, such that 
$X$ admits a surjective morphism to a smooth projective curve $C$ of genus $g \ge 2$. Then:

\noindent
(i) $Y$ admits a surjective morphism to a curve of genus $\ge g$. 

\noindent
(ii) If $X$ has a Fano fibration structure over $C$, then so does $Y$, and $X$ and $Y$
are $K$-equivalent.\footnote{Recall that this means that there exist a smooth projective $Z$ and birational morphisms
$f\colon Z \rightarrow X$ and $g\colon Z \rightarrow Y$ such that $f^* \omega_X \simeq g^* \omega_Y$.} In particular, if $X$ is a Mori fiber space 
over $C$, then $X$ and $Y$ are isomorphic.
\end{theorem}

A slightly stronger statement is given in Theorem \ref{fano}.
We remark that it is known from results of Beauville and Siu that $X$ admits a surjective morphism to a curve of 
genus $\ge g$ if and only if $\pi_1 (X)$ has a surjective homomorphism onto $\Gamma_g$, the fundamental group of a 
Riemann surface of genus $g$ (see the Appendix to \cite{catanese}). 
On the other hand, it is also known that derived equivalent varieties do not necessarily have isomorphic fundamental 
groups (see \cite{bak}, \cite{schnell}), 
so this would not suffice in order to deduce  Theorem \ref{pencils} (i). A more precise version of (i) can be found 
in the Remark on p. 302; see also Question \ref{curves}. 
The refinement we give in (ii) in the case of Fano fibrations answers a question posed to us by Y. Kawamata; for this, 
the method of proof is completely independent of the study of $V^i (\omega_X)$, relying instead of Kawamata's kernel 
technique \cite{kawamata} and on the structure of the Albanese map for varieties with nef anticanonical bundle \cite{zhang}. 
The result however fits naturally in the present context.

Going back to the main results, the isomorphism between the $V^i_0$ is realized, as in \cite{lombardi}, via the Rouquier isomorphism 
associated to a Fourier-Mukai equivalence (see \S\ref{sec cyclic}). To relate this to the behavior of 
Hodge numbers of type $h^{0,i}$ as in Theorem \ref{main}, the main new ingredients are Simpson's result 
describing the components of all $V^i (\omega_X)$ as torsion translates of abelian subvarieties of $\Pic^0 (X)$, used via a density 
argument involving torsion points of special prime order, and the comparison of the derived categories of cyclic covers 
associated to torsion line bundles mapped to each other via the Rouquier isomorphism, modeled after and slightly extending 
results of Bridgeland-Maciocia \cite{BM} on equivalences of canonical covers.

\section{Derived equivalences of cyclic covers}\label{sec cyclic}

\noindent
{\bf Cyclic covers.}
Let $X$ be a complex smooth projective variety and $\alpha$ be a $d$-torsion element of $\Pic^0 (X)$. 
We denote by $$\pi_{\alpha} \colon X_{\alpha}\rightarrow X$$ the \'{e}tale cyclic cover of order $d$ associated to 
$\alpha$ (see e.g. \cite[\S7.3]{huybrechts}). Then 
\begin{equation}\label{projection}
\pi_{\alpha *}\OO_{X_{\alpha}}\simeq \bigoplus_{i=0}^{d-1}\alpha^{-i}
\end{equation}
and there is a free action of the group $G:=\ZZ / d\ZZ$ on $X_{\alpha}$ such that $X_{\alpha}/G\simeq X$.
The following Lemma is analogous to \cite[Proposition 2.5(b)]{BM}. We include a proof for completeness,
entirely inspired by the approach in \cite[Proposition 2.5(a)]{BM}.

\begin{lemma}\label{ob}
Let $E$ be an object of $\D(X)$. There is an object 
$E_{\alpha}$ in $\D(X_{\alpha})$ such that $\pi_{\alpha*} E_{\alpha} \simeq E$ if and only if
$E\otimes \alpha\simeq E$.
\end{lemma}
\begin{proof} 
For the nontrivial implication, let 
$$s \colon E\stackrel{\simeq}{\longrightarrow} E\otimes \alpha$$ 
be an isomorphism.
We proceed by induction on the number $r$ of non-zero cohomology sheaves of $E$.
If $E$ is a sheaf concentrated in degree zero, then the Lemma is a standard fact. Indeed, it is well known that 
$$\pi_{\alpha*} \colon {\rm Coh}(X_{\alpha}) \rightarrow {\rm Coh}(\mathcal{A})$$ 
is an equivalence between the category of coherent $\OO_{X_{\alpha}}$-modules 
and the category of coherent 
$\mathcal{A} : = (\bigoplus_{i=0}^{d-1} \alpha^{i})$-algebras, while
a coherent sheaf $E$ on $X$ belongs to ${\rm Coh}(\mathcal{A})$ 
if and only if $E\otimes \alpha\simeq E$.

Suppose now that the Lemma is true for all objects having at most $r-1$ non-zero cohomology sheaves, 
and consider an object $E$ with $r$ non-zero cohomology sheaves. By shifting $E$, we can assume that 
$\mathcal{H}^i(E) = 0$ for $i\notin ~[-(r-~1),0]$. Since $E\otimes \alpha\simeq E$, 
we also have $\mathcal{H}^0(E)\otimes \alpha\simeq \mathcal{H}^0(E)$. 
Therefore, by the above, there exists a coherent sheaf 
$M_{\alpha}$ on $X_{\alpha}$ such that $\pi_{\alpha*}M_{\alpha}\simeq \mathcal{H}^0(E)$.
Now the natural morphism $E\stackrel{j}{\rightarrow} \mathcal{H}^0(E)$ induces a distinguished triangle 
$$
E\stackrel{j}{\rightarrow} \mathcal{H}^0(E)\stackrel{f}{\rightarrow} F\rightarrow E[1]
$$
such that the object $F$ has $r-1$ non-zero cohomology sheaves.
By the commutativity of the following diagram 

\centerline{ \xymatrix@=32pt{
 & E \ar^{j}[r]\ar[d] & \mathcal{H}^0(E)\ar[d]\ar[r]^{f} & F\ar[r]  & E[1] \ar[d]\\
&  E\otimes \alpha \ar^{j\otimes \alpha}[r] & \mathcal{H}^0(E)\otimes \alpha  \ar[r]^{f\otimes \alpha} 
& F\otimes \alpha \ar[r] & (E\otimes \alpha)[1],  & \\ }}
\noindent we obtain an isomorphism $F\simeq F\otimes \alpha$, and therefore by induction 
an object $F_{\alpha}$ in $\D(X_{\alpha})$ such that $\pi_{\alpha*} F_{\alpha}\simeq F$. 

To show the existence of an object $E_{\alpha}$ in $\D(X_{\alpha})$ such that $\pi_{\alpha*}E_{\alpha}\simeq E$, 
we assume for a moment that there exists a morphism $f_{\alpha}:M_{\alpha}\rightarrow F_{\alpha}$ such that $\pi_{\alpha*} f_{\alpha} =f$.
This is enough to conclude, since by completing $f_{\alpha}$ to a distinguished triangle
$$M_{\alpha}\stackrel{f_{\alpha}}{\rightarrow}F_{\alpha}\rightarrow E_{\alpha}[1]\rightarrow M_{\alpha}[1],$$
and applying $\pi_{\alpha*}$, we obtain $\pi_{\alpha*}E_{\alpha}\simeq E$.

We are left with showing the existence of $f_{\alpha}$. 
Let $\lambda_{\alpha}:\pi_{\alpha*}M_{\alpha} \rightarrow \pi_{\alpha*}M_{\alpha}\otimes ~\alpha$ and
$\mu_{\alpha}:\pi_{\alpha*}F_{\alpha} \rightarrow \pi_{\alpha*}F_{\alpha}\otimes \alpha$ 
be the isomorphisms determined by the diagram above. Note that 
\begin{equation}\label{eq}
\mu_{\alpha}\circ f= (f\otimes \alpha)\circ \lambda_{\alpha} \quad \mbox{ in }\quad \D(X).
\end{equation}
We can replace $F_\alpha$ by an injective resolution
$$\cdots \, \rightarrow \, \I_{\alpha}^{-1} \, \stackrel{d^{-1}}{\rightarrow} \, \I_{\alpha}^{0} \, \stackrel{d^0}{\rightarrow} \, \I_{\alpha}^1 \, \stackrel{d^1}
{\rightarrow} \, \cdots,$$
so that $f$ is represented (up to homotopy) by a morphism of $\OO_X$-modules
$$u:\pi_{\alpha*}M_{\alpha}\rightarrow \pi_{\alpha*}\I_{\alpha}^0.$$
Let $V$ be the image of the map
$${\rm Hom}(\pi_{\alpha*}M_{\alpha},-) \colon {\rm Hom}(\pi_{\alpha*}M_{\alpha},\pi_{\alpha*}\I_{\alpha}^{-1})\rightarrow 
{\rm Hom}(\pi_{\alpha*}M_{\alpha},\pi_{\alpha*}\I_{\alpha}^0).$$ 
By \eqref{eq}, we have isomorphisms of $\OO_X$-modules $a_1 \colon \pi_{\alpha*}M_{\alpha} \rightarrow 
\pi_{\alpha*}M_{\alpha}\otimes \alpha$ and 
$b_1\colon \pi_{\alpha*}\I_{\alpha}^0 \rightarrow \pi_{\alpha*}\I_{\alpha}^0\otimes \alpha$ such that 
\begin{equation}\label{equ}
 b_1\circ u = (u\otimes \alpha)\circ a_1 \qquad \mbox{up to homotopy.}
 \end{equation}
By setting $a_{i}:=(a_1\otimes \alpha^{i-1})\circ \cdots \circ (a_1 \otimes \alpha)\circ a_1$ (for $i\geq 2$) 
and similarly for $b_i$, we
define an action of $G:=\ZZ /d\ZZ$ on $V$ as
$$g^i\cdot (-) \, :=  \, b^{-1}_{i}\circ (-\otimes \alpha^{ i})\circ a_i,$$ 
where $g$ is a generator of $G$.
Moreover, we define operators $A$ and $B$ on $V$ as
$$A \, := \, \sum_{i=0}^{d-1}g^i\cdot(-),\quad \quad B:=1-g\cdot (-).$$
Since $AB=0$, we note that ${\rm Ker}\,A={\rm Im}\,B$.

By \eqref{equ}, we have that $B(u)=u-b_1^{-1}\circ (u\otimes \alpha)\circ a_1$ is null-homotopic, and therefore
$B(u)\in V$. Since ${\rm Ker}\,A={\rm Im}\,B$, there exists a morphism $\eta\in V$ such that $B(\eta)=B(u)$.
Consider the morphism ~$t \, := \, u-\eta\in{\rm Hom}(\pi_{\alpha*}M_{\alpha},\pi_{\alpha*}\I_{\alpha}^0)$. It is easy to check that $t$ is 
homotopic to $u$ and therefore it represents $f$ as well.
But now $B(t)=0$, so $t = \pi_{\alpha*}(v)$ for some morphism $v \colon M_{\alpha}\rightarrow 
\I_{\alpha}^0$, which concludes the proof.
\end{proof}

\noindent
{\bf Rouquier isomorphism.}
It is well known by Orlov's criterion that every equivalence $\Phi\colon \D(X) \rightarrow \D(Y)$ is of Fourier-Mukai type, 
i.e. induced by an object $\E \in  \D(X\times Y)$, unique up to isomorphism, via
$$\Phi = \Phi_{\E} \colon \D(X) \rightarrow \D(Y), \,\,\,\, \Phi_{\E} (-) = \R{p_Y}_* \big(p_X^* (-) \stackrel{\mathbf{L}}{\otimes} \E\big).$$
For every such equivalence, Rouquier \cite[Th\'eor\`eme 4.18]{rouquier} showed that there is an induced isomorphism of algebraic groups
$$F_{\E} \colon {\rm Aut}^0 (X) \times {\rm Pic}^0 (X) \rightarrow {\rm Aut}^0 (Y) \times 
{\rm Pic}^0 (Y)$$
which usually mixes the two factors. A concrete formula for $F_{\E}$ was worked out in \cite[Lemma 3.1]{PS}, namely
\begin{equation}\label{rouquier_formula}
F_{\E} (\varphi, \alpha) = (\psi, \beta) \iff p_X^{\ast} \alpha \otimes (\varphi\times \id_Y)^{\ast} \E 
\simeq p_Y^{\ast} \beta \otimes (\id_X \times \psi)_{\ast} \E.
\end{equation}

\noindent
{\bf Derived equivalences of cyclic covers.}
Before stating the main theorem of this section, we recall two definitions from \cite{BM} (see also \cite[\S7.3]{huybrechts}).
Let $\widetilde{X}$ and $\widetilde{Y}$ be two smooth projective varieties on which the group 
$G:=\ZZ/d\ZZ$ acts freely. Denote by
$\pi_X \colon  \widetilde{X}\rightarrow X$ and $\pi_Y \colon \widetilde{Y}\rightarrow Y$ the quotient maps of $\widetilde{X}$ and $\widetilde{Y}$ 
respectively.

\begin{defn}\label{equiv}
A functor $\widetilde{\Phi} \colon \D(\widetilde{X})\rightarrow \D(\widetilde{Y})$ is \emph{equivariant} if 
there exist an automorphism $\mu$ of $G$ and isomorphisms of functors
$$g^*\circ \widetilde{\Phi}\simeq \widetilde{\Phi}\circ \mu(g)^*\quad \mbox{ for all}\quad g\in G.$$ 
\end{defn}

\begin{defn}\label{lift}
 Let $\Phi \colon \D(X)\rightarrow \D(Y)$ be a functor. A \emph{lift} of $\Phi$ is a functor 
$\widetilde{\Phi} \colon \D(\widetilde{X})\rightarrow \D(\widetilde{Y})$ inducing isomorphisms
\begin{equation}\label{comm1}
 \pi_{Y*} \, \circ \, \widetilde{\Phi} \, \simeq \, \Phi \, \circ  \, \pi_{X*}
\end{equation}
\begin{equation}\label{comm2}
\pi^*_Y \, \circ \, \Phi \, \simeq \, \widetilde{\Phi} \, \circ \, \pi_X^*.
\end{equation}
\end{defn}

\emph{Remark.} 
If $\Phi \colon \D(X)\rightarrow \D(Y)$ and $\widetilde{\Phi}  \colon  \D(\widetilde{X})\rightarrow \D(\widetilde{Y})$ 
are equivalences, then by taking the adjoints \eqref{comm1} holds if and only if \eqref{comm2} holds.

Now we are ready to prove the main result of this section. It is a slight extension of the result of 
\cite{BM} on canonical covers, whose proof almost entirely follows the one given there, and which serves as a 
technical tool for our main theorem.

\begin{theorem}\label{cyclic}
Let $X$ and $Y$ be smooth projective varieties, and $\alpha\in\Pic^0 (X)$ 
and 
$\beta\in\Pic^0 (Y)$ $d$-torsion elements. 
Denote by $\pi_{\alpha} \colon X_{\alpha}\rightarrow X$ and
$\pi_{\beta} \colon ~Y_{\beta}\rightarrow Y$ the cyclic covers associated to $\alpha$ and $\beta$ respectively.

\noindent
(i) Suppose that $\Phi_{\E}\colon \D(X)\rightarrow \D(Y)$ is an equivalence, and that $F_{\E}(\id_X,\alpha)=(\id_Y,\beta)$. 
Then there exists an equivariant 
equivalence $\Phi_{\widetilde\E} \colon \D(X_{\alpha})\rightarrow \D(Y_{\beta})$ lifting $\Phi_{\E}$.

\noindent
(ii) Suppose that $\Phi_{\widetilde{\F}} \colon \D(X_{\alpha})\rightarrow \D(Y_{\beta})$ is an equivariant equivalence. 
Then $\Phi_{\widetilde{\F}}$ is the lift of some equivalence $\Phi_{\F} \colon \D(X)\rightarrow \D(Y)$.
\end{theorem}
\begin{proof}
To see (i), consider the following commutative diagram, where $p_1$, $p_2$, $r_1$ and $r_2$ are projection maps:

\centerline{ 
\xymatrix@=32pt{
 X_{\alpha} \ar[d]^{\pi_{\alpha}} &  X_{\alpha}\times Y \ar[l]_{r_1 \,\,\,\,}\ar[d]^{\pi_{\alpha}\times \id_Y} \ar[r]^{\,\,\,\,\,\,\,\, r_2} & Y\ar@{=}[d] \\
 X  & \ar[l]_{p_1 \,\,\,\,}  X\times Y  \ar[r]^{\,\,\,\,\,\,\,\, p_2}  & Y.\\ }}
\noindent By (\ref{rouquier_formula}), the condition $F_{\E}(\id_X,\alpha)=(\id_Y,\beta)$ is equivalent to the isomorphism 
in $\D(X\times Y)$:
\begin{equation}\label{kernel}
p_1^*\alpha\otimes \E\simeq p_2^*\beta\otimes \E.
\end{equation}
Pulling \eqref{kernel} back via the map $(\pi_{\alpha}\times \id_Y)$, we get an isomorphism
$$(\pi_{\alpha}\times \id_Y)^*\E\simeq  r_2^*\beta\otimes (\pi_{\alpha}\times \id_Y)^*\E$$ as $\pi_{\alpha}^*\alpha\simeq \OO_{X_{\alpha}}$.
As the map $(\id_{X_{\alpha}}\times \pi_{\beta}) \colon X_{\alpha}\times Y_{\beta}\rightarrow X_{\alpha}\times Y$ is the 
\'{e}tale cyclic cover associated to the line bundle $r_2^*\beta$,
by Lemma \ref{ob} there exists an object $\widetilde{\E}$ such that 
\begin{equation*}
(\id_{X_{\alpha}}\times \pi_{\beta})_*\widetilde{\E}\simeq (\pi_{\alpha}\times \id_Y)^*\E.
\end{equation*}
By \cite{BM} Lemma 4.4, there is an isomorphism
\begin{equation}\label{iso1}
\pi_{\beta*} \, \circ \, \Phi_{\widetilde{\E}} \, \simeq \, \Phi_{\E} \, \circ \, \pi_{\alpha*}.
\end{equation}

We now show that $\Phi_{\widetilde{\E}}$ is an equivalence.
Let $\Psi_{\E'} \colon \D(Y)\rightarrow \D(X)$ be a quasi-inverse of $\Phi_{\E}$. Since $F_{\E'}=F^{-1}_{\E}$, we have that
$F_{\E'}(\id_X,\beta)=(\id_Y,\alpha)$.
By repeating the previous argument, one then sees that there exists an object $\widetilde{\E'}$ such that
$$(\pi_{\alpha} \times \id_{Y_{\beta}})_*\widetilde{\E'} \simeq (\id_X \times \pi_{\beta})^*\E' $$ 
and an isomorphism of functors
\begin{equation}\label{iso2}
\pi_{\alpha*} \, \circ \, \Psi_{\widetilde{\E'}} \, \simeq \, \Psi_{\E'} \, \circ \, \pi_{\beta*}.
\end{equation}
Since $\Psi_{\E'}\circ \Phi_{\E}\simeq \id_{\D(X)}$, using \eqref{iso1} and \eqref{iso2} we get an isomorphism
\begin{equation}\label{iso3}
\pi_{\alpha*} \, \circ \, \Psi_{\widetilde{\E'}} \, \circ  \, \Phi_{\widetilde{\E}} \, \simeq \, \Psi_{\E'} \, \circ \, \pi_{\beta*} \, \circ \, \Phi_{\widetilde{\E}} \, 
\simeq \, \Psi_{\E'} \, \circ \, \Phi_{\E} \, \circ \, \pi_{\alpha*} \, \simeq \, \pi_{\alpha*}.
\end{equation}
Hence, following the proof of \cite[Lemma 4.3]{BM}, we have that $\Psi_{\widetilde{\E'}}\circ \Phi_{\widetilde{\E}}\simeq g_*(L\otimes -)$ 
for some $g\in G$ and $L \in {\rm Pic} (X_\alpha)$. 
By taking left adjoints in \eqref{iso3}, we obtain on the other hand that 
$$(L^{-1} \otimes -)\circ g^* \circ \pi_{\alpha}^*\simeq 
\pi_{\alpha}^*,$$ 
which applied to $\OO_X$ yields $L\simeq \OO_{X_{\alpha}}$. 
This gives $\Psi_{\widetilde{\E'}}\circ \Phi_{\widetilde{\E}} \simeq g_*$.
Similarly, we can show that $\Phi_{\widetilde{\E}}\circ \Psi_{\widetilde{\E'}}\simeq h_*$ 
for some $h\in G$, and hence that 
$g^* \circ \Psi_{\widetilde{\E'}}$, or equivalently $\Psi_{\widetilde{\E'}}\circ h^*$, 
is a quasi-inverse of $\Phi_{\widetilde{\E}}$. 
Finally, Remark on p. 297 implies that $\Phi_{\widetilde{\E}}$ is a lift of $\Phi_{\E}$.

The proofs of the fact that $\Phi_{\widetilde{\E}}$ is equivariant 
and of (ii) are now completely analogous to those of the corresponding statements in \cite[Theorem 4.5]{BM}.
\end{proof}

\section{Comparison of cohomological support loci}\label{sec comparison}

A more precise statement than that of Conjecture \ref{var1} naturally involves the dimension of  the
cohomology groups of the line bundles mapped to each other via the Rouquier isomorphism. 
Such a statement was proved in \cite{lombardi} when $i = 0, 1$. The concrete statement, which also 
contains Conjecture \ref{special_hodge} by specializing at the origin, is the following:

\begin{conj}\label{stronger}
Let $X$ and $Y$ be smooth projective varieties of dimension $n$, and let $\Phi_{\E}$ be 
a Fourier-Mukai equivalence between $\D(X)$ and $\D(Y)$.
If $F = F_{\E}$ is the induced Rouquier isomorphism, then 
$$F \left(\id_X, V^i (\omega_X)_0 \right) = \left(\id_Y, V^i (\omega_Y)_0\right)$$
for all $i$, so that $V^i (\omega_X)_0 \simeq V^i (\omega_Y)_0$. 
Moreover if $\alpha \in V^i (\omega_X)_0$ and $F(\id_X,  \alpha) = (\id_Y, \beta)$, then
$$h^i (X, \omega_X \otimes \alpha)  = h^i (Y, \omega_Y \otimes \beta).$$
\end{conj}

In order to address this statement, we consider for each $m \ge 1$ the more refined cohomological support 
loci 
$$V^i_m( \omega_X) : = \{ \alpha \in \Pic^0 (X)~|~ h^i (X, \omega_X \otimes \alpha) \ge m\}.$$
In this notation $V^i (\omega_X)$ becomes  $V^i_1 (\omega_X)$.
The following result is a strengthening of Theorem \ref{main} in the Introduction.\footnote{We are grateful to the 
referee, who suggested that our original argument for $m = 1$ applies in fact to all $m$. This was crucial for deducing 
this more refined statement, as opposed to just that of Conjecture \ref{var1}, from
Conjecture \ref{special_hodge}.}

\begin{theorem}\label{main_improved}
Conjecture \ref{special_hodge} is equivalent to Conjecture \ref{stronger}. 
Specifically, if $n = {\rm dim}~X$, Conjecture \ref{special_hodge} for $n-i$ implies 
$$F \left(\id_X, V^i_m (\omega_X)_0 \right) = \left(\id_Y, V^i_m (\omega_Y)_0\right)$$
for all $m \ge 1$.
\end{theorem}
\begin{proof}
Note first that $F$ induces an isomorphism on the locus of line bundles $\alpha \in \Pic^0 (X)$ with the 
property that $F (\id_X, \alpha) = (\id_Y, \beta)$ for some $\beta \in \Pic^0 (Y)$, so indeed the first 
assertion in the statement follows from the second one, which we prove in a few steps.

\noindent
\emph{Step 1.}
We first show that if $\alpha \in V^i (\omega_X)_0$, then it does satisfy the property above, namely
there exists $\beta \in \Pic^0 (Y)$ such that 
$$F (\id_X, \alpha) = (\id_Y, \beta).$$
A more general statement has already been proved in \cite[Theorem 3.2]{lombardi}. We  
extract the argument we need here in order to keep the proof self-contained, following \cite[\S3]{PS} as well. 
The Rouquier isomorphism $F$ induces a morphism
$$\pi \colon \Pic^0(Y) \to \Aut^0 (X), \quad \pi(\beta)= p_1 \bigl( F^{-1}(\id_Y, \beta) \bigr),$$
whose image is an abelian variety $A$ and where $p_1$ is the projection from ${\rm Aut}^0(X)\times {\rm Pic}^0(X)$ onto the 
first factor. 
If $A$ is trivial there is nothing to prove, so we can assume that $A$ is 
positive dimensional.

As $A$ is an abelian variety of automorphisms of $X$, according to \cite[\S3]{brion} 
there exists a finite subgroup
$H \subset A$ and an \'etale locally trivial fibration
$p\colon X \rightarrow A/H$ which is trivialized by base change to $A$.
In other words, there is  a cartesian diagram
$$\xymatrix{
A \ar[r]^{\,\,\,\, g}  \ar[d] \times Z & X \ar[d]^{p \,\,\,\,\,\,\,\,}\\
A \ar[r] & A/H
}
$$
where $Z = p^{-1}(0)$. Restricting $g$ to the fiber $A\times \{z_0\}$ where $z_0$ 
is an arbitrary point in $Z$, we obtain a morphism $f \colon  A \rightarrow X$, which is 
in fact an orbit of the action of $A$ on $X$. It is shown in the proof of \cite[Theorem A]{PS}, that 
${\rm Ker}(f^*)^0\simeq {\rm Ker}(\pi)^0$, where $(\cdot)^0$ denotes the connected component of the identity; this is based on 
a theorem of Matsumura-Nishi, essentially saying that the induced $f^*\colon  \Pic^0 (X) \rightarrow \Pic^0 (A)$ is surjective.
Consequently, we only need to show that $\alpha \in {\rm Ker}(f^*)$; it will then automatically be in ${\rm Ker}(f^*)^0$, 
since it lives in $V^i (\omega_X)_0$, which is a union of abelian subvarieties.

To this end, note that $\alpha \in V^i (\omega_X)_0$ implies that 
$$H^i \left(A \times Z, g^*(\omega_X\otimes \alpha)\right) \simeq 
H^i \left(A \times Z, f^* \alpha \boxtimes (\omega_Z \otimes \alpha_{|Z})\right) \neq 0.$$
Applying the K\"unneth formula, we conclude that we must have
$$H^k (A, f^* \alpha) \neq 0\,\,\,\,{\rm for~some~}  0 \le k \le i,$$
which implies that $f^* \alpha \simeq \OO_A$.

\noindent
\emph{Step 2.}
Since one can repeat the argument in Step $1$ for $F^{-1}$, it is then enough to show that 
if $\alpha \in V^i_m (\omega_X)_0$ and $F(\id_X, \alpha) = (\id_Y , \beta)$, then 
$\beta\in V^i_m (\omega_Y)_0$ as well.
In this step we show that it is enough to prove this assertion in the case when $\alpha\in \Pic^0 (X)$ 
is a torsion point of (special) prime order. 
First, since $F$ is a group isomorphism, $\alpha$ is torsion of some order if and only if $\beta$ is torsion of the same order.

According to a well-known theorem of Simpson \cite{simpson}, every irreducible component $Z$ of $V^i_m (\omega_Y)$ is a torsion translate 
$\tau_Z + A_Z$ of an abelian subvariety of $\Pic^0 (Y)$. We consider the set $P_i$ of all prime 
numbers that do not divide ${\rm ord}(\tau_Z)$ for any such component $Z$. As 
$V^i_m (\omega_Y)$ is an algebraic set by the semicontinuity theorem, 
we are only throwing away a finite set of primes.
We will show that it is enough to prove the assertion above when $\alpha$ is torsion 
with order in $P_i$.
First note that it is a standard fact that torsion points of prime order are Zariski dense in a 
complex abelian variety.\footnote{This follows for instance 
from the fact that real numbers can be approximated with rational numbers with prime denominators.} Consequently, 
torsion points with order in the set $P_i$ are dense  as well.

Let now $W$ be a component of $V^i_m (\omega_X)_0$. It suffices to show that 
$$Z : = p_2 \left( F(\id_X, W)\right) \subset V^i_m (\omega_Y)_0,$$
where $p_2$ is the projection onto the second component of $\Aut^0 (Y) \times \Pic^0 (Y)$. 
Indeed, since one can repeat the same argument for the inverse homomorphism $F^{-1}$, 
this implies that $Z$ has to be a component of $V^i_m (\omega_Y)_0$, isomorphic to $W$ via 
$F$. Now $Z$ is an abelian variety, and therefore by the discussion above torsion points $\beta$ of order in $P_i$ are 
dense in $Z$. By semicontinuity, it suffices to show that $\beta \in V^i_m (\omega_Y)_0$.
These $\beta$'s are precisely the images of $\alpha\in W$ of order in $P_i$, which concludes our reduction step.

\noindent
\emph{Step 3.}
Let now $\alpha \in V^i_m (\omega_X)_0$ be a torsion point of order belonging to the set $P_i$, and 
$F(\id_X, \alpha) = (\id_Y, \beta)$. Denote 
$$p = {\rm ord}(\alpha) = {\rm ord}(\beta).$$
Consider the cyclic covers $\pi_{\alpha} \colon X_{\alpha}\rightarrow X$ and $\pi_{\beta} \colon Y_{\beta}\rightarrow Y$ 
associated to $\alpha$ and $\beta$ respectively. We can apply Theorem \ref{cyclic} to conclude that 
there exists a Fourier-Mukai equivalence
$$\Phi_{\widetilde \E} \colon \D(X_\alpha) \rightarrow \D(Y_\beta)$$
lifting $\Phi_\E$. Assuming Conjecture \ref{special_hodge}, we have in particular that
$$h^{0, n-i} (X) = h^{0, n-i} (Y) \,\,\,\,{\rm and} \,\,\,\, h^{0, n-i} (X_\alpha) = h^{0, n-i} (Y_\beta).$$ 
On the other hand, using (\ref{projection}), we have
$$H^{n-i} (X_\alpha, \OO_{X_{\alpha}}) \simeq \bigoplus_{j=0}^{p-1}H^{n-i} (X, \alpha^{-j})
\,\,\,\,{\rm and}$$ $$H^{n-i} (Y_\beta, \OO_{Y_{\beta}}) \simeq \bigoplus_{j=0}^{p-1}H^{n-i} (Y, \beta^{-j}).$$
The terms on the left hand side and the terms corresponding to $j = 0$ on the 
right hand side have the same dimension. 
On the other hand, since every component of $V^i_m (\omega_X)_0$ is an abelian subvariety of $\Pic^0 (X)$, 
we have that $\alpha^j \in V^i_m (\omega_X)_0$ for all $j$, so  
$$h^{n-i} (X, \alpha^{-j}) \ge m \,\,\,\, {\rm for~ all~} j.$$
We conclude that 
$$h^{n-i} (Y, \beta^{-k}) \ge m \,\,\,\,{\rm for~some} \,\,\,\, 1 \le k \le p-1.$$
This says that $\beta^k \in V^i_m (\omega_Y)$. 
We claim that in fact $\beta^k \in V^i_m (\omega_Y)_0$. 
Assuming that this is the case, we can conclude 
the argument. Indeed, pick a component $T \subset V^i_m (\omega_Y)_0$ such that $\beta^k \in T$.
But $\beta^k$ generates the cyclic group of prime order $\{1, \beta, \ldots, \beta^{p-1}\}$, so $\beta \in T$ as well, 
since $T$ is an abelian variety.

We are left with proving that $\beta^k \in V^i_m (\omega_Y)_0$. Pick any component $S$ in $V^i_m (\omega_Y)$
containing $\beta^k$. By the Simpson theorem mentioned above, 
we have that $S = \tau + B$, where $\tau$ is a torsion point and $B$ is an abelian 
subvariety of $\Pic^0 (Y)$. We claim that we must have $\tau \in B$, so that $S = B$, confirming our 
statement.\footnote{Note that in fact we are proving 
something stronger: $\beta^k$ belongs \emph{only} to components of $V^i_m (\omega_Y)$ 
passing through the origin.} To this end, switching abusively to additive notation, 
say $k \beta = \tau + b$ with $b \in B$, and denote the torsion order of $\tau$ by $r$. 
Since the order $p$ of $\beta$ is assumed to be in the set $P_i$, we have that $r$ and $p$ are coprime. 
Now on one hand $r\tau = 0 \in B$,  while on the other hand  $p\tau + pb = kp\beta = 0$, so $p\tau \in B$
as well. Since $r$ and $p$ are coprime, one easily concludes that $\tau \in B$.
\end{proof}

\section{Fibrations over curves}\label{sec fibrations}

\noindent
{\bf Fibration structure via derived equivalence.}
We now apply the derived invariance of $V^{n-1}(\omega_X)_0$ to deduce Theorem 
\ref{pencils} (i) in the Introduction.

\begin{proof}(of Theorem \ref{pencils} (i)).
Let $f \colon X \rightarrow C$ be a surjective morphism onto a smooth projective curve of genus $g \ge 2$. 
Using Stein factorization, we can assume that $f$ has connected fibers. 
We have that $f^* \Pic^0 (C) \subset V^{n-1} (\omega_X)_0$. Since by Corollary \ref{non-vanishing} we 
have $V^{n-1} (\omega_X)_0 \simeq V^{n-1} (\omega_Y)_0$, there exists a component $T$ of $V^{n-1} (\omega_Y)_0$ of dimension at 
least $g$. By \cite[Corollaire 2.3]{beauville}, there exists a smooth projective curve $D$ and a surjective morphism with 
connected fibers $g\colon Y \rightarrow D$ such that $T = g^* \Pic^0 (D)$. Note that $g(D) = {\rm dim}~T \ge g$.
\end{proof}

\emph{Remark.} 
The discussion above shows in fact the following more refined statement. For a smooth projective variety $Z$, define
$$A_Z : = \{ g \in \NN ~|~ g = {\rm dim}~T {\rm ~for~some~irreducible ~component}~ T \subset V^{n-1}(\omega_Z)_0\}.$$
Then if $\D(X) \simeq \D(Y)$, we have $A_X = A_Y$. Denoting this set by $A$, for each $g\in A$ both $X$ and $Y$ have 
surjective maps onto curves of genus $g$. The maximal genus of a curve admitting a surjective map from $X$ (or $Y$) is ${\rm max}(A)$.

\begin{question}\label{curves}
If $\D(X) \simeq \D(Y)$, is the set of curves of genus at least $2$ admitting non-constant 
maps from $X$ the same as that for $Y$? Or at least the set of curves corresponding to irreducible components of $V^{n-1}(\omega_X)_0$?
\end{question}

\noindent
{\bf Fano fibrations.}
The following is a slightly more precise version of Theorem \ref{pencils} (ii) in the Introduction.

\begin{theorem}\label{fano}
Let $X$ and $Y$ be smooth projective complex varieties such that $\D(X) \simeq \D(Y)$. 
Assume that there is an algebraic fiber space $f \colon X \rightarrow C$ such that $C$ is a smooth projective curve of genus at least $2$ and 
the general fiber of $f$ is Fano. Then:

\noindent
(i)  $X$ and $Y$ are $K$-equivalent.

\noindent
(ii) There is an algebraic fiber space $g\colon Y\rightarrow C$ such that for $c \in C$  where the fibers $X_c$ and $Y_c$ are 
smooth, with $X_c$ Fano, one has $Y_c \simeq X_c$.

\noindent
(iii) If $\omega_X^{-1}$ is $f$-ample (e.g. if $f$ is a Mori fiber space), then $X \simeq Y$.
\end{theorem}
\begin{proof}
Let $p$ and $q$ be the projections of $X\times Y$ onto the first and second factor respectively. 
Consider the unique up to isomorphism $\E\in \D(X\times Y)$ such that the given equivalence is the Fourier-Mukai functor $\Phi_{\E}$. 
Then by \cite[Corollary 6.5]{huybrechts}, there 
exists a component $Z$ of ${\rm Supp} (\E)$ such that $p_{|Z} \colon Z \rightarrow X$ is surjective. We first claim that 
${\rm dim}~Z = {\rm dim}~X$.

Assuming by contradiction that ${\rm dim}~Z > {\rm dim}~X$, we show that $\omega_X^{-1}$ is nef. 
We denote by $F$ the general fiber of $f$, which is Fano. We also define
$Z_F : = ~p_{Z}^{-1} (F) \subset Z$, while $q_F \colon Z_F \rightarrow Y$ is the projection obtained by restricting $q$ to $Z_F$. 
Since $\omega_F^{-1}$ is ample, we obtain that $q_F$ is finite onto its 
image; see \cite[Corollary 6.8]{huybrechts}. 
On the other hand, the assumption that ${\rm dim}~Z > {\rm dim}~X$ implies that ${\rm dim}~Z_F \ge {\rm dim}~X = \dim Y$, so 
$q_F$ must be surjective (and consequently ${\rm dim}~Z_F = {\rm dim}~X$). 

By passing to its normalization if necessary, we can assume without loss of generality that $Z_F$ is normal. Denoting 
by $p_F$ the projection of $Z_F$ to $X$, by \cite[Corollary 6.9]{huybrechts} we have that there exists $r > 0$ such that 
$$p_F^* ~\omega_X^{-r} \simeq q_F^*~ \omega_Y^{-r}.$$
Now since $p_F$ factors through $F$ and $\omega_F^{-1}$ is ample, we have 
that $p_F^* \omega_X^{-1}$ is nef, hence by the isomorphism above so is $q_F^*~ \omega_Y^{-1}$. Finally, since $q_F$ is finite and 
surjective, we obtain that $\omega_Y^{-1}$ is nef, so by  
\cite[Theorem 1.4]{kawamata}, $\omega_X^{-1}$ is nef as well.

We can now conclude the proof of the claim using the main result of Zhang 
\cite{zhang} (part of 
a conjecture of Demailly-Peternell-Schneider), 
saying that a smooth projective variety with nef anticanonical bundle
has surjective Albanese map. In our case, since the general fiber of $f$ is Fano, the Albanese map of $X$ is obtained
by composing $f$ with the Abel-Jacobi embedding of $C$. But this implies that $C$ has genus at most $1$, a contradiction.
The claim is proved, so
$$\dim Z = \dim X = \dim Y.$$ 
At this stage, the $K$-equivalence statement follows from Lemma \ref{K_equiv} below.

For statements (ii) and (iii) we emphasize that, once we know that $X$ and $Y$ are $K$-equivalent, the argument is standard and 
independent of derived equivalence.\footnote{We thank Alessio Corti for pointing this out to us.} Note first that smooth birational 
varieties have the same Albanese variety and Albanese image. Since $f$ is the Albanese map of $X$, it follows that the Albanese map 
of $Y$ is a surjective morphism $g\colon  Y \rightarrow C$. 
Furthermore, $C$ is the Albanese image of any other birational model as well, hence any smooth model $Z$ inducing a $K$-equivalence 
between $X$ and $Y$ sits in a commutative diagram
$$\xymatrix{
& Z \ar[dl]_p \ar[dr]^q \ar[dd]_{h} & \\
X \ar[dr]^{f } &  & Y \ar[dl]_{ g}    \\
& C & 
}
$$
Note that in particular $g$ has connected fibers since $f$ does.

For a point $c \in C$, denote by $X_c$, $Y_c$ and $Z_c$ the fibers of $f$, $g$ and $h$ over $c$. 
By adjunction, $Z_c$ realizes a $K$-equivalence between $X_c$ and $Y_c$. 
First, assuming that $c$ is chosen such that $X_c$ and $Y_c$ are smooth, with $X_c$ Fano,  
we show that $X_c \simeq Y_c$. 

To this end, if we assume that the induced rational map $\varphi_c \colon Y_c \rightarrow X_c$ is not a morphism, there must 
be a curve $B \subset Z_c$ which is  contracted by $q_c$ but not by $p_c$.  
Then $q_c^* \omega_{Y_c} \cdot B = 0$, and so $p_c^* \omega_{X_c} \cdot B = 0$ as well. 
On the other hand, $\omega_{X_c}^{-1} \cdot p_c(B) < 0$ which is a contradiction. Therefore 
we obtain that $\varphi_c$ is a birational morphism with the property that $\varphi_c^* \omega_{X_c} \simeq \omega_{Y_c}$, 
which implies that $\varphi_c$ is an isomorphism.

If in fact $\omega_X^{-1}$ is $f$-ample, this argument can be globalized: indeed, assuming that the rational map 
$\varphi \colon Y \rightarrow X$ is not a morphism, there exists a curve $B \subset ~ Z$ which is contracted by $q$ and hence $h$, 
 but not by $p$. Since $B$ lives in a fiber of $f$ (by the commutativity of the diagram), we again obtain a contradiction. 
 Once we know that $\varphi$ is a morphism, the same argument as above implies that it is an isomorphism.
\end{proof}

The following Lemma used in the proof above is due to Kawamata, and can be extracted from his argument 
leading to the fact that derived equivalent varieties of general type are $K$-equivalent \cite{kawamata}; we sketch the argument for convenience.

\begin{lemma}\label{K_equiv}
Let $\Phi_{\E} \colon  \D(X) \rightarrow \D(Y)$ be a derived equivalence, and 
assume that there exists a component $Z$ of the support
of $\E$ such that $\dim Z = \dim X$ and $Z$ dominates $X$. Then $X$ and $Y$ are $K$-equivalent.
\end{lemma}
\begin{proof}
Denote by $p$ and $q$ the projections of ${Z}$ to $X$ and $Y$.
Since $p$ is surjective, \cite[Corollary 6.12]{huybrechts} tells us that $p$ is birational, and 
$Z$ is the unique component of ${\rm Supp}(\E)$ dominating $X$. 
We claim that $q$ is also surjective, in which case by the same reasoning $q$ is birational as well. 
Since (on the normalization of $Z$) we have $p^* \omega_X^r \simeq q^* \omega_Y^r$ for some $r\ge 1$, this 
suffices to conclude that $X$ and $Y$ are $K$-equivalent as in \cite[Theorem 2.3]{kawamata} (see also \cite[p.149]{huybrechts}).

Assuming that $q$ is not surjective, we can find general points $x_1$ and $x_2$ in $X$ such that $p^{-1} (x_1)$ 
and $p^{-1} (x_2)$ consist of one point, and 
$q (p^{-1} (x_1)) = q (p^{-1} (x_2)) = y$ for some $y \in Y$. One then sees that 
$${\rm Supp} ~\Phi_{\E} (\OO_{x_1}) = {\rm Supp} ~\Phi_{\E} (\OO_{x_2}) = \{y\}.$$
This implies in standard fashion that 
$${\rm Hom}^{\bullet}_{\D(X)} (\OO_{x_1}, \OO_{x_2}) \simeq {\rm Hom}^{\bullet}_{\D(Y)} 
(\Phi_{\E} (\OO_{x_1}), \Phi_{\E} (\OO_{x_2})) \neq 0,$$
a contradiction.
\end{proof}

\section*{Acknowledgments}
We thank Yujiro Kawamata and Anatoly Libgober for discussions that motivated 
this work, and Caucher Birkar, Alessio Corti, Lawrence Ein, Anne-Sophie Kaloghiros, Artie Prendergast-Smith and 
Christian Schnell for answering numerous questions. 
We are grateful to the referee, who suggested how to improve the statement of Theorem \ref{main} in the form
of Theorem \ref{main_improved}.
Special thanks go to Rob Lazarsfeld, our first and second generation teacher, whose fundamental 
results in the study of cohomological support loci have provided the original inspiration for this line of research.
MP was partially supported by the NSF grant DMS-1101323.

\bibliographystyle{amsalpha}

\begin{thebibliography}{13}


\bibitem{bak}
Bak, A. 2009.
The spectral construction for a $(1,8)$-polarized family of abelian varieties.
Preprint {\em arXiv:0903.5488}.

\bibitem{beauville}
Beauville, A. 1992. 
Annulation du H$^1$ pour les fibr\'es en droites plats.
Pages 1--15 of: {\em Complex algebraic varieties (Proc. Bayreuth 1990), Vol. 1507}. Lectures Notes in Math., 
Springer-Verlag, Berlin.

\bibitem{BM}
Bridgeland, T., and Maciocia, A. 1998.
Fourier-Mukai transforms for quotient varieties.
Preprint {\em arXiv:9811101}.

\bibitem{BM2}
Bridgeland, T., and Maciocia, A. 2001.
Complex surfaces with equivalent derived categories. 
{\em Math. Z.}, {\bf 236}, no. 4, 677--697.

\bibitem{brion}
Brion, M. 2009.
Some basic results on actions of non-affine algebraic groups.
Pages 1--20 of: {\em Symmetry and Spaces: in honor of Gerry Schwarz, Vol. 278}. Progr. Math., 
Birkhauser Boston Inc.  

\bibitem{catanese}
Catanese, F. 1991.
Moduli and classification of irregular K\"ahler manifolds (and algebraic varieties) with Albanese general type fibrations. 
{\em Invent. Math.}, {\bf 104}, no. 2, 389--407. 

\bibitem{gl}
Green, M., and Lazarsfeld, R. 1991.
Higher obstructions to deforming cohomology groups of line bundles.
{\em J. Amer. Math. Soc.}, {\bf 1}, no. 4, 87--103.

\bibitem{huybrechts}
Huybrechts, H. 2006.
Fourier-Mukai transforms in algebraic geometry. 
Oxford Mathematical Monographs. 
The Clarendon Press Oxford University Press.

\bibitem{kawamata}
Kawamata, Y. 2002. 
$D$-equivalence and $K$-equivalence. 
{\em J. Diff. Geom.}, {\bf 61}, no. 1, 147--171.

\bibitem{lombardi}
Lombardi, L. 2014. 
Derived invariants of irregular varieties and Hochschild homology. 
{\em Algebra \& Number Theory}, {\bf 8}, no. 3, 513--542

\bibitem{popa}
Popa, M. 2013.
Derived equivalence and non-vanishing loci.
{\em Clay Math. Proceedings volume}, A celebration of Algebraic Geometry, volume in honor of Joe Harris' $60$th birthday, {\bf 18}, p. 567--575, Amer. Math. Soc., Providence, RI, 

\bibitem{PS}
Popa, M., and Schnell, C. 2011.
Derived invariance of the number of holomorphic $1$-forms and vector fields.  
{\em Ann. Sci. ENS}, {\bf 44}, no. 3, 527--536. 

\bibitem{rouquier}
Rouquier, R. 2011.
Automorphismes, graduations et cat\'egories triangul\'ees.
{\em J. Inst. Math. Jussieu}, {\bf 10}, no. 3, 713--751.

\bibitem{schnell}
Schnell, C. 2012. 
The fundamental group is not a derived invariant.
Pages 279--285 of:
{\em Derived categories in algebraic geometry}. EMS Ser. Congr. Rep.,
Eur. Math. Soc., Z\"{u}rich.

\bibitem{simpson}
Simpson, C. 1993.
Subspaces of moduli spaces of rank one local systems.
{\em Ann. Sci. ENS}, {\bf 26}, no. 3, 361--401.

\bibitem{toda}
Toda, T. 2006.
Fourier-Mukai transforms and canonical divisors.
{\em Compositio Math.}, {\bf 142}, no. 4, 962--982.

\bibitem{zhang}
Zhang, Q. 1996.
On projective manifolds with nef anticanonical bundles. 
{\em J. Reine Angew. Math.}, {\bf 478}, 57--60.


\end{thebibliography}

\end{document}